\documentclass[12pt]{article}
\usepackage{amsmath,amssymb,latexsym,color,epsfig}

\newsavebox{\theorembox}
\newsavebox{\lemmabox}
\newsavebox{\corollarybox}
\newsavebox{\propositionbox}
\newsavebox{\examplebox}
\newsavebox{\conjecturebox}
\newsavebox{\algbox}
\newsavebox{\qbox}
\newsavebox{\problembox}
\newsavebox{\definitionbox}
\newsavebox{\assumptionbox}
\newsavebox{\hypothesisbox}
\savebox{\theorembox}{\noindent\bf Theorem}
\savebox{\lemmabox}{\noindent\bf Lemma}
\savebox{\corollarybox}{\noindent\bf Corollary}
\savebox{\propositionbox}{\noindent\bf Proposition}
\savebox{\examplebox}{\noindent\bf Example}
\savebox{\conjecturebox}{\noindent\bf Conjecture}
\savebox{\algbox}{\noindent\bf Algorithm}
\savebox{\qbox}{\noindent\bf Question}
\savebox{\definitionbox}{\noindent\bf Definition}
\savebox{\problembox}{\noindent\bf Problem}
\savebox{\assumptionbox}{\noindent\bf Assumption}
\savebox{\hypothesisbox}{\noindent\bf Hypothesis}

\newtheorem{theorem}{\usebox{\theorembox}}[section]
\newtheorem{lemma}{\usebox{\lemmabox}}[section]
\newtheorem{corollary}{\usebox{\corollarybox}}[section]
\newtheorem{proposition}{\usebox{\propositionbox}}[section]

\newtheorem{definition}{\usebox{\definitionbox}}

\newcommand{\qed}{\;\;\;\Box}
\newenvironment{proof}{\par{\bf Proof:}}{\(\qed\) \par}
\newcommand{\Bin}{\mathrm{Bin}}

\begin{document}

\title{Contagious Sets in Random Graphs}

\author{Uriel Feige\thanks{Department of Computer Science and Applied Mathematics, the Weizmann Institute, Rehovot, 7610001, Israel. {\tt
    uriel.feige@weizmann.ac.il}. Work supported in part by the Israel Science Foundation
   (grant  621/12), and by the I-CORE Program of the Planning and Budgeting Committee and the Israel Science Foundation (grant  4/11)}.
\and Michael Krivelevich\thanks{School of Mathematical Sciences, Raymond and Beverly Sackler Faculty of Exact Sciences, Tel Aviv University, 6997801, Israel. {\tt krivelev@post.tau.ac.il }.
    Research supported in part by: the USA-Israel BSF (grants 2010115, 2014361), and by the Israel Science Foundation (grant 912/12).}
        \and Daniel Reichman\thanks{Institute of Cognitive and Brain Sciences, University of California, Berkeley, CA. {\tt
    daniel.reichman@gmail.com}. Supported in part by the Israel Science Foundation (grant  621/12).}}

\maketitle

\begin{abstract}
We consider the following activation process in undirected graphs: a vertex is active either if it belongs to a set of initially activated vertices or if at some point it has at least $r$ active neighbors. A \emph{contagious set} is a set whose activation results with the entire graph being active. Given a graph $G$, let $m(G,r)$ be the minimal size of a contagious set.

We study this process on the binomial random graph $G:=G(n,p)$ with $p: = \frac{d}{n}$ and $1 \ll d \ll \left(\frac{n \log \log n}{\log^2 n}\right)^{\frac{r-1}{r}}$. Assuming $r > 1$ to be a constant that does not depend on $n$, we prove that $$m(G,r) = \Theta\left(\frac{n}{d^{\frac{r}{r-1}}\log d}\right),$$ with high probability.
We also show that the threshold probability for $m(G,r)=r$ to hold is $p^*=\Theta\left(\frac{1}{(n \log^{r-1} n)^{1/r}}\right)$.
\end{abstract}
\newpage

\section{Introduction}

In $r$-\emph{neighbor bootstrap percolation} we are given an undirected graph $G=(V,E)$ and an integer $r>1$. Every vertex is either \emph{active} or \emph{inactive}. A set of vertices composed entirely of active vertices is called active. Initially, a set of vertices $A_0$ is activated. These vertices are called \emph{seeds}. A contagious process evolves in discrete steps where for $i>0$,
$$A_i=A_{i-1}\cup \{v:|N(v)\cap A_{i-1}|\geq r\},$$
and $N(v)$ is the set of neighbors of $v$. In words, a vertex becomes active in a given step if it has at least $r$ active neighbors. We refer to $r$ throughout this paper as the \emph{threshold}. Set
$$\langle A_0 \rangle:=\bigcup_iA_i.$$
\begin{definition}
Given $G=(V,E)$ and a threshold $r$, a set $A_0\subseteq V$ is called \emph{contagious} if $\langle A_0 \rangle=V$. That is, activating $A_0$ results with the entire graph being activated.
The minimal cardinality of a contagious set in $G$ is denoted in $G$ by $m(G,r)$.
The \emph{number of generations} of a (not necessarily contagious) set $A_0$ which we denote by $\tau:=\tau(A_0)$ is the minimal integer such that
$\bigcup_{i \leq \tau}A_i=\langle A_0 \rangle$.
\end{definition}

Bootstrap percolation has been studied for a variety of graphs \cite{BB,Lattice,peres,BP,Janson,JanLuc}. Here we focus on the random graph $G(n,p)$ on $n$ labeled vertices, where every possible edge appears independently with probability $p$. Our interest is in providing both upper and lower bounds on the typical size of a contagious set of minimal cardinality. We remark that the term ``bootstrap percolation" is often used with respect to choosing vertices independently with some probability $q$ to the set of seeds. In contrast, in this work we do not restrict ourselves to the study of randomly generated contagious sets.

Studying the behavior of combinatorial quantities in $G(n,p)$ has a long and rich history \cite{Boll}, and has resulted in a plethora of ideas which have proven useful in other contexts as well.
In addition, there is much interest in studying computational problems on random graphs \cite{frieze}.
Furthermore, combinatorial and algorithmic ideas originating from the study of the model $G(n,p)$ of random graphs are often useful in the study of more general families of random graphs. Hence, beyond the intrinsic value of studying the value of $m(G,r)$ in $G(n,p)$ which we consider to be of interest of its own right, we believe the ideas in the current work may prove applicable in other
contexts where contagious processes are studied.

\subsection{Our results}

Consider $G(n,p)$, and let $p:=\frac{d}{n}$. We obtain a nearly tight characterization of the probable value of $m(G,r)$. We say an event in the probability space $G(n,p)$ occurs ``typically" or ``with high probability" (w.h.p.) if it occurs with probability $1-o(1)$, where $o(1)$ represents a term that tends to $0$ as $n$ tends to infinity. For two integer valued function $f(n),g(n)$, we say that $f(n) \ll g(n)$ if $\lim_{n\rightarrow \infty} \frac{f(n)}{g(n)}=0.$ 

\begin{theorem}
\label{thm:random}
Let $G\sim G(n,p)$ with $p: = \frac{d}{n}$ and $$1 \ll d \ll \left(\frac{n \log \log n}{\log^2 n}\right)^{\frac{r-1}{r}}.$$ 
Then with high probability
$$m(G,r)= \Theta\left(\frac{n}{d^{\frac{r}{r-1}}\log d}\right).$$
\end{theorem}

The upper bound in Theorem~\ref{thm:random} is constructive in the sense that it is derived by analyzing a polynomial time algorithm that typically finds a contagious set of size at most $O\left(\frac{n}{d^{\frac{r}{r-1}}\log d}\right)$.

Clearly it is always the case that $m(G,r) \geq r$. We examine how large $p$ needs to be in order for $G(n,p)$ to satisfy that typically $m(G,r)=r$.  The property of having a contagious set of size $r$ is a monotone property, hence it has a sharp threshold \cite{Bollo3}. We determine this threshold up to constant multiplicative factors:

\begin{theorem}
\label{thm:threshold}
Let $G\sim G(n,p)$ and suppose $r \geq 2$ is an integer. There exist $0<c<C$, such that the following holds: if $p<\frac{c}{(n \log^{r-1} n)^{1/r}}$, then with high probability no set of size $r$ is contagious.  If $p>\frac{C}{(n \log^{r-1} n)^{1/r}}$,
then with high probability there are contagious sets of size $r$. Moreover, 
with high probability there is a choice of a contagious set $B_0$ of size $r$ for which $\tau(B_0)=O(\log \log n)$. This upper bound on $\tau(B_0)$ is best possible up to constant factors -- with high probability there is no contagious set $B$ of size $r$ with $\tau(B)=o(\log \log n)$, as long as $p = o(n^{-1/r})$.
\end{theorem}

\subsection{Related work}

Bootstrap percolation was introduced by Chalupa, Leath and Reich \cite{Chal}, motivated by applications in statistical physics. Other early works include \cite{Aizenman,Enter}. Initially, the study of bootstrap percolation  focused mostly on lattices and grids. More recently, it has been studied on other families of graphs such as random $d$-regular graphs \cite{BP,Janson}, hypercubes \cite{BB} and several models of random graphs with a given degree sequence (e.g., \cite{Amini,AF}). These works studied the case in which the set of seeds is selected independently at random. The smallest contagious set (the value of $m(G,2)$) was studied for some  families of graphs such as hypercubes \cite{BB} and grids \cite{square}.

The critical size of a \emph{random} set needed for full activation (with high probability) of the binomial random graph $G(n,p)$ was first studied in \cite{Vallier}. The results in \cite{Vallier} were generalized and extended by \cite{JanLuc} (using ideas from \cite{Tomba}),
where the critical size of a random set required for complete activation of $G(n,p)$ for arbitrary constant threshold $r$  is determined in great detail of precision. We shall apply the following theorem from \cite{JanLuc} (which follows from Theorem 3.1, page 1996, and Theorem 3.10, page 2000, in \cite{JanLuc}).

\begin{theorem}\label{thm:Janson}
Let $r \geq 2$ be a fixed integer independent of $n$. Suppose $G\sim G(n,p)$ with $n^{-1}\ll p \ll n^{-1/r}$. Let $$a_c:=\left(1-\frac{1}{r}\right)\cdot \left(\frac{(r-1)!}{np^r}\right)^{1/(r-1)}.$$ Suppose that $A$ is a fixed set of vertices that are activated as seeds. Then for every fixed $\delta>0$, with high probability the following holds:

\begin{enumerate}
\item If $|A| = (1+\delta)a_c$ then at least $n-O(n(pn)^{r-1}e^{-pn})$ vertices will be infected. Furthermore, $\tau(A)=\frac{\ln \ln(np)}{\ln r}+\frac{\ln n}{np}+O(1)$.

\item If $|A| \le (1-\delta)a_c$ then at most $2\left(\frac{(r-1)!}{np^r}\right)^{1/(r-1)}$ vertices will be infected. 

\end{enumerate}

\end{theorem}

For example, Theorem~\ref{thm:Janson} implies that when $G\sim G(n,p)$ with $p$ as above, then with high probability $m(G,2)\leq \frac{1+\delta}{2np^2}$. (Observe that $n^2pe^{-pn}=o(n/d^2)$ for the range of $p$ in Theorem~\ref{thm:Janson}, and hence the set of vertices not activated by $A$ is small and can be added to the set of seeds with only negligible effect on the total number of seeds.)
To the best of our knowledge, the upper bound $m(G,r)\leq (1+\delta)a_c$ was the best upper bound known on $m(G,r)$ in random graphs prior to our work.

The lower bound of Theorem~\ref{thm:Janson} implies that a \emph{randomly chosen} set of $(1-\delta)a_c$ vertices has only negligible probability of being contagious.
Our upper bound in Theorem~\ref{thm:random} (whose proof involves a more sophisticated choice of set of seeds) implies that for such graphs $m(G,r)$ is with high probability significantly smaller than $a_c$. 
This shows that choosing an initial set of seeds carefully (rather than uniformly at random) is typically beneficial for this key model of random graphs.

It is proven in \cite{JanLuc} that when $p \gg n^{-1/r}$, an arbitrary set of size $r$ of activated vertices will infect the whole of $G(n,p)$ w.h.p. Similarly to Theorem~\ref{thm:random}, Theorem~\ref{thm:threshold} demonstrates that a careful choice of the seeds results in a contagious set of size $r$ for $p$ much smaller than $n^{-1/r}$.

Theorem~\ref{thm:random} and the constructive nature of the upper bound there imply that there is a polynomial time algorithm that for most graphs (from the distribution specified in Theorem~\ref{thm:random}) returns a contagious set whose size is within a constant factor of the minimum possible. In contrast, on worst-case instances, approximating the minimal size of a contagious set within a ratio better than $O(2^{\log^{1-\delta}n})$ ($n$ is the number of vertices) is intractable for every $\delta \in (0,1)$, unless NP $ \subseteq$ DTIME$(n^{poly(\log n)})$ \cite{Chen09}.

Theorem~\ref{thm:Janson} (taken from~\cite{JanLuc}) considers also $\tau$, the number of generations until complete activation. The parameter $\tau$ has been studied also in families of graphs such as grids \cite{Bolo1,Bolo2} and dense graphs \cite{polo}. We consider $\tau$ in the context of Theorem~\ref{thm:threshold} but not in the context of Theorem~\ref{thm:random}. We briefly discuss $\tau$ further in Section~\ref{sec:concl}.

The minimal number of edges that forces an $n$-vertex graph to satisfy $m(G,r)=r$ was considered in \cite{polo}. For example, it is proven that a graph having at least ${n-1 \choose 2}+1$ edges must satisfy $m(G,2)=2$. This result is tight, as $m(G,2)=3$ (for $n \geq 3$) when $G$ is a clique on $n-1$ vertices along with an additional isolated vertex.

The current paper is one part of a larger body of work whose preliminary version is available in \cite{Expanders}. Other parts of that work will be published separately, and they concern contagious sets in $d$-regular graphs. For example, it is shown there that sufficiently strong expansion properties (e.g., spectral gap $d-O(\sqrt{d})$, or girth $\Omega(\log \log d)$) ensure that $m(G,2) \leq O(n/d^2)$, where $n=(V(G)|$. (Recall in contrast that the best general upper bound for the value of $m(G,2)$ for $d$-regular graphs on $n$ vertices is $m(G,2)\le \frac{2n}{d+1}$ \cite{Re12}; this bound is easily seen to be tight.) In addition, it is shown that when $G$ is a random $d$-regular graph over $n$ vertices (which with high probability is an excellent spectral expander, see \cite{friedman03}), it holds that $m(G,2) \geq \Omega(\frac{n}{d^2\log d})$ with high probability. That lower bound regarding random $d$-regular graphs is established using ideas similar to those used to establish the lower bound in Theorem~\ref{thm:random}.

\subsection{Overview of proof techniques}
\label{sec:overview}

The proof of the upper bound in Theorem \ref{thm:random} is based on the following observation (we consider $r=2$ throughout this section -- similar reasoning applies for $r>2$).
For a subset $A\subseteq V$, we denote by $N(A)$ the set of all vertices in $V \setminus A$ having a neighbor in $A$.
Suppose we have an initial set $A$ of seeds, and consider $N(A)$. Given that the graph is random, one can analyze the distribution of the sizes of the connected components of the subgraph induced by $N(A)$.
Introducing a single seed in a connected component of size $k$ then activates the whole component, thus giving $k$ activated vertices per investment of one seed.
It turns out that we can activate a set of size $\frac{n}{d^2} = \frac{1}{np^2}$ in $G$ by choosing $O\left(n\frac{\log\log d}{d^2 \log d}\right)$ seeds in this way. Thereafter, the results of~\cite{JanLuc} can be used in order to deduce that $G$ (apart from a set of negligible cardinality which can be activated separately) is activated with high probability.

To achieve the improved upper bound in Theorem~\ref{thm:random}, we repeat the procedure above iteratively. In iteration~0, choose an arbitrary set $A_0$ of seeds of size $\frac{n}{d^2\log d}$. Next, for each $1 \le i \le \log\log d$, consider the external neighborhood of the vertices activated in iteration $i-1$. Within this neighborhood identify the largest connected components, and activate a set $B_i$ that includes one vertex from each component (thus infecting the whole component), until the sum of sizes of infected components reaches $\frac{2^i n}{d^2\log d}$. After $\log\log d$ iterations we have $\frac{n}{d^2}$ active vertices, which as previously noted suffices to infect the whole of $G$ (apart from a set of negligible cardinality treated separately). The total number of activated vertices is $|A_0| + \sum_{i=1}^{\log\log d} |B_i|$. We show that the latter sum is bounded by $O(\frac{n}{d^2\log d})$, with high probability.

Our lower bound in Theorem~\ref{thm:random} is based on observing that if there is a contagious set of size $t_0$, then adding to it the first $t-t_0$ infected vertices gives an induced subgraph with $t$ vertices and at least $2(t-t_0)$ edges. For a choice of $t_0 < \frac{n}{6d^2 \log d}$ and $t= \frac{n}{3d^2}$, a simple probabilistic argument shows that a random graph with high probability does not contain any such subgraph.

For Theorem~\ref{thm:threshold}, the proof of the lower bound on the threshold probability for $m(G,2)=2$ follows the same principles as the lower bound for Theorem~\ref{thm:random} (but with $t_0 = 2$).
For the proof of the upper bound (the typical existence of a contagious set of size $2$ when $p\ge \frac{C}{\sqrt{n \log n}}$) we represent $G$ as a union of two random graphs $G_1$ and $G_2$ with edge probabilities $\frac{1}{\sqrt{n \log n}}$  and $\frac{C_1}{\sqrt{n \log n}}$, respectively.
We first show that in $G_1$, a random set of two vertices has probability significantly higher than $\frac{\log n}{n}$ of infecting $\Omega(\log n)$ additional vertices. We then show that this implies that with high probability, there is at least one pair of vertices that infects a set $S$ of size $\Omega(\log n)$ in $G_1$. Finally, the results of \cite{JanLuc} are used to prove that with high probability $S$ will infect the whole of $G_2$, and thus the whole of $G$.

\subsection{Preliminaries and notation}

Let $H=(V,E)$ be an undirected graph. For $A,B\subseteq V$, we define $E(A)$ to be the set of all edges spanned by $A$ and $E(A,B)$ the set of all edges with one endpoint in $A$ and one endpoint in $B$.  The notation $\log$ denotes logarithms in base $2$ and $\ln$ denotes natural logarithms. The set of integers $\{1 \ldots \ell\}$ for $\ell \geq 1$ is denoted by $[\ell]$. We will reserve the notation $G$ for $G(n,p)$ throughout this paper, omitting the dependency on $n,p$ when clear from the context.

We shall use the term {\em infected} vertex to describe an activated vertex that is not one of the {\em seeds}, but has rather become activated by having at least $2$ active neighbors. We say a set $S \subseteq V$ is \emph{infected} if all the vertices of $S$ are infected.

We close this section with a version of Chernoff's inequality (see, e.g., \cite{MU}).
\begin{lemma}
  \label{lemma:Chernoff}
  Suppose that $X = \sum_{i=1}^mX_i$, where every $X_i$ is a $\{0,1\}$-random variable with $\Pr(X_i=1)=p$ and the $X_i$s are jointly independent. Then for arbitrary $\eta \in (0,1)$, it holds that

  \[
  \Pr(X<(1-\eta)pm) \le \exp(-pm\eta^2/2),
  \]

  and
  \[
  \Pr(X>(1+\eta)pm) \le \exp(-pm\eta^2/3).
  \]
  \end{lemma}
  \subsection{Organization}
  We first present our results when $r=2$ as this case is more transparent, making it easier to present the main ideas behind the proofs. In Section~\ref{sec:two} we prove that with high probability $m(G,2)=\Theta(\frac{n}{d^2 \log d})$, dealing first with the upper bound and then establishing a lower bound. In Section~\ref{sec:thresholdtwo} we determine the asymptotic threshold of having a contagious set of size $2$. In Section~\ref{sec:generalization} we discuss how to generalize the results of Sections~\ref{sec:two} and \ref{sec:thresholdtwo} to the case where $r>2$. In Section \ref{sec:concl} we present some concluding comments.

\section{$m(G,2)$ in random graphs}\label{sec:two}
In this section we prove Theorem~\ref{thm:random} for the case $r=2$. Unless explicitly stated, we will always focus on $G(n,p)$ where $p: = \frac{d}{n}$ is as in the range of Theorem~\ref{thm:random}.

\subsection{Upper bound}
\bigskip

The following lemma can be derived from known results (e.g., \cite{Boll}) but we present a self-contained proof for completeness.
\begin{lemma}
\label{lemma:subcritical}
Let $H:=G(n_0,q)$ be the binomial random graph with $n_0$ vertices and edge probability $q$ (we assume $n_0$ is large enough). Let $k=O(\log n_0)$ be an integer and $q=\frac{c}{n_0}$.
Then for every $c<1/20$, the probability a given vertex $v$ belongs to a connected component of size at least $k$ is at least $\left(\frac{c}{3}\right)^{k-1}$. Furthermore, with probability at least $1-\exp{(-\Omega(\frac{n_0}{k}\left(\frac{c}{3}\right)^{k-1}))}$ the number of vertices lying in components of size at least $k$ is at least $\left(\frac{c}{3}\right)^{k-1}\cdot n_0/4$.
\end{lemma}
\begin{proof}
Consider the following iterative procedure of exposing edges in $H$. Every vertex has a mark: either it is \emph{used} or it is \emph{unused}. In the beginning of the algorithm, all vertices are marked as unused and all edges of $H$ are not exposed.
We continue the process as long as the number of unused vertices is at least $\frac{n_0}{2}+k$ (or, differently put, the number of used vertices is at most $\frac{n_0}{2}-k)$.

In the beginning of every iteration, we choose an unused vertex $v$ in $H$, and attempt to expose a simple path in $H$ containing $v$ as follows. Consider a set $S$ of exactly $n_0/2$ unused vertices. Expose all edges between $v$ and $S$. If there is a vertex $v_2 \in S$ connected to $v$, add $v_2$ to the path. Continue in this fashion (attaching a vertex to the last vertex appended to the path) until either one of two cases occurs: a {\em success}, meaning that the size of the path containing $v$ reaches $k$, or a {\em failure}, meaning that we have failed to find a path of $k$ vertices containing $v$ (namely, we constructed a path $P$ of $l<k$ vertices, and the last vertex on the path has no edge to any of the vertices of the corresponding set $S$). Finally, proceed by marking all the vertices that are in the path rooted at $v$ as used.

The probability that all vertices in a set $U$ of $n_0/2$ unused vertices are not connected to a vertex $w \notin U$ is $(1-q)^{n_0/2}$. It follows that the probability that  during an iteration we can append a new vertex to a path of length smaller than $k$ is exactly
$$(1-(1-q)^{n_0/2}) \geq \frac{n_0q}{3},$$
and this holds \emph{independently} of the length of the path we have constructed thus far.
Hence the probability
we succeed in growing a path of length $k$ (and hence in a connected component of size at least $k$) at a given iteration is at at least
$$\left( \frac{n_0q}{3}\right)^{k-1} = \left( \frac{c}{3}\right)^{k-1}.$$

As $\lfloor n_0/(3k) \rfloor \leq\frac{\frac{n_0}{2}-k}{k}$ (recall we assume $k=O(\log n_0)$), it follows that the distribution of the number of successes (until less than $\frac{n_0}{2}+k$ unused vertices remain) stochastically dominates the binomial distribution with $\lfloor n_0/(3k) \rfloor$ trials and success probability $\left(\frac{c}{3}\right)^{k-1}$ (the exact number of trials depends on the number of failures, but failures only increase the number of trials). Furthermore, standard concentration results concerning the binomial distribution imply that probability at least $1-\exp{(-\Omega(\frac{n_0}{k}\left(\frac{c}{3}\right)^{k-1}))}$, the number of successes is at least $\frac{n_0}{4k}\left(\frac{c}{3}\right)^{k-1}$. Since every success places $k$ vertices (rather than just one) in a component of size at least $k$ the lemma is proven.
\end{proof}

\bigskip

We shall also rely on the following lemma.
\begin{lemma}\label{lemma:partial_infec}
Let the activation threshold be $r=2$ and let $d_0$ be a sufficiently large constant. Then for $d=d(n)\ge d_0$, a random graph $G\sim G(n,p)$ is w.h.p. such that activating any set of size at least $n/2$ infects all but at most $n/d^3$ vertices.
\end{lemma}

\begin{proof}
Let $A_0$ be an initially activated set, and let $U=[n]-\langle A_0\rangle$. Then $|U|\le n/2$, and every vertex of $U$ has at most one neighbor outside of $U$, implying that the number of edges crossing between $U$ and its complement is at most $|U|$. The probability of having such a set $U$ of cardinality $|U|\ge n/d^3$ in $G(n,p)$ can be estimated from above through Lemma \ref{lemma:Chernoff} as follows:
\begin{gather*}
\sum_{k=max\{1,n/d^3\}}^{n/2} \binom{n}{k}Pr [\Bin(k(n-k),p)\le k]
 \le \sum_{k=\max\{1,n/d^3\}}^{n/2} \binom{n}{k} e^{-\frac{knp}{8}}\\
\le \sum_{k=\max\{1,n/d^3\}}^{n/2} \left(\frac{en}{k}\right)^k e^{-\frac{knp}{8}}
\end{gather*}
We now distinguish between two cases: if $d \le n^{1/3}$ then we upper bound the summation above by
$$\sum_{k=n/d^3}^{n/2} (ed^{1/3}\cdot e^{-\frac{np}{8}})^k=o(1).$$
Otherwise, if $d>n^{/3}$ the summation can be upper bounded by
$$\sum_{k=1}^{n/2}(en\cdot e^{-n^{2/3}/8})^k=o(1),$$
as desired.

\end{proof}
\bigskip

\begin{theorem}\label{thm:main}
If $d=np$ satisfies
$1 \ll d \ll \left(\frac{n\log\log n}{\log^2 n}\right)^{1/2}$ 
and $G\sim G(n,p)$, then whp $m(G,2)\le \frac{13n}{d^2\log d}$.
\end{theorem}

We give a constructive proof for Theorem~\ref{thm:main}. Namely, we provide an algorithm that finds a contagious set that is not larger than the upper bound in this Theorem.
Our algorithm is composed of three stages described below.

{\bf Stage I.} Set
$$
\ell = \log\log d\,.
$$
Initialize $B_0=C_0=D_0$ to be a fixed subset of $[n]$ of size $\frac{n}{d^2\log d}$.

For $i=1,\ldots,\ell$ repeat:

Set
$$
s_i = \frac{\log d-4}{\ell-i+4}\,.
$$
\noindent{\it Step (i1).} Expose edges of $G$ between $C_{i-1}$ and $V\setminus\bigcup_{j=0}^{i-1}B_j$. Let $B_i$ be an arbitrary set of $\frac{d|C_{i-1}|}{2}$ neighbors of  $C_{i-1}$ in $V\setminus\bigcup_{j=0}^{i-1}B_j$. If there is no such set -- declare a failure;

\noindent{\it Step (i2).} Expose edges of $G$ inside $B_i$. Let $x_i$ be the number of connected components of $G[B_i]$ and define
$y_i=\min\left\{x_i,\frac{n}{d^22^{\ell-i}s_i}\right\}$. Let $T_{i1},\ldots,T_{iy_i}$ be the $y_i$ largest components of $G[B_i]$ (breaking ties arbitrarily). If
$$
|\bigcup_{j=1}^{y_i}T_{ij}|<\frac{n}{d^22^{\ell-i}}
$$
-- declare a failure. Otherwise form $D_i$ by choosing one arbitrary vertex from each $T_{ij}$. Clearly
$$
|D_i|=y_i\le \frac{n}{d^22^{\ell-i}s_i}\,.
$$
Let $C_i$ be an arbitrary subset of $\bigcup_{j=1}^{y_i}T_{ij}$ of size
$$
|C_i|=\frac{n}{d^22^{\ell-i}}\,.
$$

\medskip

Assume that Stage I was successful for every $1 \le i \le \ell$. Denote
$$
A_{01}=D_0\cup D_1\cup\ldots\cup D_\ell\,.
$$
The algorithm activates all vertices in $A_{01}$. Finally let $A_{02}$ the set of vertices that remain inactive after $A_{01}$ is activated. We activate all vertices in $A_{02}$
We now prove a series of propositions that upper bound the size of $A_{01} \bigcup A_{02}$.
\begin{proposition}
Activating $A_{01}$ infects $\bigcup_{j=0}^{\ell}C_j$.
\end{proposition}
\begin{proof}
We prove by induction that activating $D_0\cup\ldots D_i$ infects $\bigcup_{j=0}^iC_j$. Induction basis follows from the definition of $C_0,D_0$. For the induction step, assume that $C_{i-1}$ is already infected. Recall that each vertex in $T_{i1},\ldots,T_{iy_i}$ has a neighbor in $C_{i-1}$ by the definition of $B_i$. Activating in addition the vertex $v$, where $\{v\}=D_i\cap T_{ij}$, infects all of $T_{ij}$, implying that all of $C_i$ gets infected.
\end{proof}
\bigskip

Let us now estimate the probability of failure of each round of Stage I, and see what it delivers assuming its success. Notice first that for all $i \in [l]$ we have that $|C_i|\le \frac{n}{d^2}$, $|B_i|\le \frac{d}{2}|C_{i-1}|$, implying that during Stage I the union $B_0\cup B_1\cup\ldots$ always has cardinality at most $(\ell+1)\cdot d\cdot n/d^2\le n/10$.

Next observe that if there does not exist an index $i \in [\ell]$ for which a failure occurs in steps $i1,i2$, then by the definition of the algorithm above we have that for every $i$ the following equalities hold : $|C_i|=\frac{n}{d^22^{\ell-i}}$ and $|B_i|=\frac{n}{d\cdot 2^{\ell-i+2}}$ which implies that, $|D_i|\le \frac{n}{d^22^{\ell-i}s_i}$ holds for every $i \in [\ell]$ as well.  Indeed, observe that if $x_i<\frac{n}{d^22^{\ell-i}s_i}$, then the union of $T_{ij}$ is the whole set $B_i$, and thus all of $B_i$ will be infected. Otherwise $y_i=\frac{n}{d^22^{\ell-i}s_i}$. If the $y_i$ largest components of $G[B_i]$ do not contain all vertices in components of size at least $s_i$, then $|T_{ij}|\ge s_i$ for $j=1,\ldots,y_i$, implying $|\bigcup_{j=1}^{y_i}T_{ij}|\ge \frac{n}{d^22^{\ell-i}}$; in the opposite case we also have the same outcome. Thus Step (i2), if successful, results indeed in a subset $C_i$ of cardinality $|C_i|=\frac{n}{d^22^{\ell-i}}$, as declared.
\begin{proposition}\label{prop:first}
With high probability there is no $i \in [\ell]$ such that step $i1$ fails.
\end{proposition}
\begin{proof}
For Step (i1), the probability that $C_{i-1}$ has less than $\frac{d}{2}|C_{i-1}|$ neighbors outside of $\bigcup_{j=0}^{i-1}B_j$ is bounded from above by
\begin{gather*}
Pr\left[\Bin\left(\frac{9n}{10},1-\left(1-\frac{d}{n}\right)^{|C_{i-1}|}\right)\le \frac{d|C_{i-1}|}{2}\right]=
\exp\left\{-\Theta(d|C_{i-1}|)\right\}\\
=\exp\left\{-\Theta\left(\frac{n}{d\cdot 2^{\ell-i}}\right)\right\}\,,
\end{gather*}
and the sum of these estimates for $i=1,\ldots,\ell$ is obviously $o(1)$.
\end{proof}
\begin{proposition}\label{prop:second}
With high probability, there is no $i \in [\ell]$ such that step $i2$ fails.
\end{proposition}
\begin{proof}
Apply Lemma~\ref{lemma:subcritical} with parameters
$$
n_0=|B_i|=\frac{n}{d\cdot 2^{\ell-i+2}},\quad q=\frac{d}{n}=\frac{1}{2^{\ell-i+2}n_0},\quad k=s_i\,.
$$
We derive that with probability $1-\exp\left\{-\Omega\left(\frac{n_0}{s_i}\left(\frac{1}{2^{\ell-i+4}}\right)^{s_i}\right)\right\}$
the set $B_i$ has at least $\left(\frac{1}{3\cdot2^{\ell-i+2}}\right)^{s_i}\frac{n}{d\cdot 2^{\ell-i+4}}$ vertices in connected components of size at least $s_i$. The (absolute value of the) exponent in the exceptional probability above can be estimated as follows:
\begin{eqnarray*}
\frac{n_0}{s_i}\left(\frac{1}{2^{\ell-i+4}}\right)^{s_i}&=&\frac{n(\ell-i+4)}{2^{\ell-i+2}d(\log d-4)}\cdot
2^{-\frac{(\ell-i+4)(\log d-4)}{\ell-i+4}}\\
&\ge& \frac{4n}{d^2\log d}\cdot\frac{\ell-i}{2^{\ell-i}}\,.
\end{eqnarray*}
Set $K:=\frac{4n}{d^2\log d}$, and observe that the requirement $d \ll \left(\frac{n\log\log n}{\log^2 n}\right)^{1/2}$ implies that $K \gg \frac{\log n}{\log\log n}$.
By the calculations above, we can upper bound the probability there is failure in one of the rounds by
$$\sum_{j=1}^{\ell}\exp\left\{-\frac{K\cdot j}{2^j}\right\}.$$ Denoting the $j$th summand by $f(j)$ we see that $\frac{f(j+1)}{f(j)}=\exp(\frac{K}{2^j}\frac{j-1}{2})$ which is
super-constant for all $1 \le j \le \ell = \log\log d$ and $n$ large enough. Therefore $\sum_{j=1}^{\ell}f(j)=\Theta(f(\ell))=\exp\left\{-\Theta(\frac{n}{d^2\log d}\cdot\frac{\log\log d}{\log d})\right\}$.
Hence recalling our assumed upper bound on $d(n)$, the union bound implies that except for probability $o(1)$, step $(i2)$ is completed for every $i \in [\ell]$. Given that there are no failures in step $(i2)$, the number of vertices of $B_i$ in components of size at least $s_i$ is at least
\begin{eqnarray*}
\left(\frac{1}{3\cdot2^{\ell-i+2}}\right)^{s_i}\frac{n}{d\cdot 2^{\ell-i+4}}
&\ge& \left(\frac{1}{2^{\ell-i+4}}\right)^{s_i}\frac{n}{d\cdot 2^{\ell-i+4}}=2^{-\log d+4}\cdot \frac{n}{d\cdot 2^{\ell-i+4}}\\
&=&\frac{n}{d^2\cdot 2^{\ell-i}}\,.
\end{eqnarray*}
\end{proof}

Now we estimate the size of the set $A_{01}=\bigcup_{i=0}^{\ell}D_i$. Recall that $|D_0|=\frac{n}{d^2\log d}$, and using Propositions~\ref{prop:first} and \ref{prop:second} we get that with probability $1-o(1)$,
$|D_i|\le\frac{n}{d^22^{\ell-i}s_i}=\frac{n(\ell-i+4)}{d^22^{\ell-i}(\log d-4)}$. It thus follows that
\begin{eqnarray*}
|A_{01}|&=&\left|\bigcup_{i=0}^{\ell}D_i\right|\le \frac{n}{d^2\log d}+\frac{n}{d^2(\log d-4)}\sum_{i=1}^{\ell}\frac{\ell-i+4}{2^{\ell-i}}\\
&=&\frac{n}{d^2\log d}+\frac{n}{d^2(\log d-4)}\left[4\sum_{i=1}^{\ell}2^{-\ell+i}+\sum_{i=1}^{\ell}\frac{\ell-i}{2^{\ell-i}}\right]\,.
\end{eqnarray*}
Obviously $\sum_{i=1}^{\ell}2^{-\ell+i}<2$. Also, $\sum_{i=1}^{\ell}\frac{\ell-i}{2^{\ell-i}}\le \sum_{j=1}^\infty \frac{j}{2^j} = \sum_{i=1}^{\infty} \sum_{j=i}^{\infty} \frac{1}{2^j} = \sum_{i=1}^{\infty}  \frac{2}{2^i} = 2$.
Altogether,
$$
|A_{01}|=\left|\bigcup_{i=0}^{\ell}D_i\right|\le \frac{n}{d^2\log d}+\frac{n}{d^2(\log d-4)}(8+2)\le \frac{12n}{d^2\log d}\,.
$$

To complete the analysis of Stage I observe that assuming it was successful, the set $C_{\ell}$ is infected, and no edges between $C_{\ell}$ and $V-\bigcup_{i=0}^{\ell}B_i$ and inside $V-\bigcup_{i=0}^{\ell}B_i$ have been exposed.

\bigskip

{\bf Stage II.} Denote $G_2=G[C_{\ell}\cup\left(V-\bigcup_{i=0}^{\ell}B_i\right)]$. We can view $G_2$ as a random graph with edge probability $p$, in which the initial seed $C_{\ell}$ of size $|C_{\ell}|=\frac{n}{d^2}$ is activated. Then according to Theorem~\ref{thm:Janson}, w.h.p. all but $O(nde^{-d/2})<\frac{n}{4}$ vertices of $G_2$ will be infected. Recalling that $\left|\bigcup_{i=0}^{\ell} B_i\right|\le \frac{n}{10}$, we arrive at the conclusion that w.h.p. after Stage II at least $n/2$ vertices of $G$ are infected, when activating the initial seed $A_{01}$.

\bigskip

{\bf Stage III.} According to Lemma~\ref{lemma:partial_infec} above, the random graph $G\sim G(n,p)$ is w.h.p. such that activating any set of size $n/2$ results in all but at most $n/d^3$ vertices being infected. Apply this Lemma to the outcome of Stage II, and denote by $A_{02}$ the set of non-infected vertices, w.h.p $|A_{02}|\le n/d^3$. Define
$$
A_0=A_{01}\cup A_{02}\,,
$$
then $|A_0|\le \frac{12n}{d^2\log d}+\frac{n}{d^3}<\frac{13n}{d^2\log d}$, and $\langle A_0\rangle=[n]$.

\subsection{Lower bound}

In our analysis, we shall include two parameters $\alpha$ and $\beta$ that can simultaneously be optimized to give the best possible lower bound provable with our current approach. For simplicity of the presentation, rather than optimizing $\alpha$ and $\beta$, we shall fix $\alpha = 3$ and $\beta = 2 - \frac{1}{\log d}$.

Let $G$ be a random graph sampled from $G(n,p)$. Let $t = \frac{n}{\alpha d^2}$. We assume that $d$ is bounded from below by some sufficiently large constant (that can be computed explicitly from the proof of Lemma~\ref{lem:density}), and bounded from above by $o(\sqrt{n})$.

\begin{lemma}
\label{lem:density}
For the setting above, w.h.p. $G$ does not have a subgraph with $t = \frac{n}{\alpha d^2}$ vertices and $\beta t$ edges, where $\alpha = 3$ and $\beta = 2 - \frac{1}{\log d}$.
\end{lemma}

\begin{proof}
There are ${n \choose t}\le (e\alpha d^2)^t$ possible choices of a set $T$ of $t$ vertices in $G$. There are ${{t \choose 2} \choose \beta t} \le (\frac{et}{2\beta})^{\beta t}$ ways of choosing $\beta t$ edge locations in $T$. The probability that all these choices are indeed edges is $\left(\frac{d}{n}\right)^{\beta t} = (\frac{1}{\alpha d t})^{\beta t}$. Hence the probability that $G$ has a subgraph with $t$ vertices and $\beta t$ edges is upper bounded by:

$$(e\alpha d^2)^t \left(\frac{et}{2\beta}\right)^{\beta t} \left(\frac{1}{\alpha d t}\right)^{\beta t} = \left(\frac{e^{\beta + 1}d^{2 - \beta}}{\alpha^{\beta - 1}2^{\beta}\beta^{\beta}}\right)^t.$$

Now in the exponent for $d$ substitute $\beta = 2 - \frac{1}{\log d}$, obtaining $d^{2 - \beta} = 2$. For the other terms we can substitute an approximation $\beta \simeq 2$, because for sufficiently large $d$, the error introduced by this is offset by our choice of $\alpha$ that is larger than needed for the proof. The expression $\frac{e^{\beta + 1}d^{2 - \beta}}{\alpha^{\beta - 1}2^{\beta}\beta^{\beta}}$ is then roughly $\frac{2e^3}{16\alpha}$ and is strictly smaller than~1 for $\alpha = 3$.
Raising to the power of $t$, the probability tends to~0 as $n$ grows.
\end{proof}

\begin{corollary}
For the parameters as above, $m(G,2) > \frac{n}{6 d^2 \log d}$ w.h.p.
\end{corollary}

\begin{proof}
Suppose otherwise. Then for $t = \frac{n}{3 d^2}$, the set of $t_0 = \frac{n}{6 d^2 \log d}$ seeds and first $t - t_0$ infected vertices induces a subgraph with $t$ vertices and $2(t - t_0) = (2 - \frac{1}{\log d})t$ edges, contradicting Lemma~\ref{lem:density}.
\end{proof}

\section{The asymptotic threshold for $m(G,2)=2$}\label{sec:thresholdtwo}

\begin{lemma}\label{lemma:threshold}
Let $p < \frac{c}{\sqrt{n \log n}}$ for some sufficiently small $c>0$. Then with high probability, $m(G,2)>2$.
\end{lemma}
\begin{proof}
Otherwise, there are two vertices $a,b$ and a set of $t \leq n-2$ vertices disjoint from $\{a,b\}$ such that the subgraph spanned on $G[\{a,b\}\cup T]$ spans at least $2t$ edges.
The probability such a subgraph exists is upper bounded by
$${n \choose 2}{n \choose t}{(t+2)^2/2 \choose 2t}p^{2t},$$ which (for large $t$) is at most
$$n^2\left(en/t\right)^t(e(t+2)p)^{2t}\leq n^2(25c)^{t}=o(1),$$
when $t=\log n$ and $c$ is sufficiently small.
\end{proof}

\medskip

We will now prove that if $p=\frac{C}{\sqrt{n\log n}}$ and $C$ is large enough, then w.h.p. $G\sim G(n,p)$ satisfies $m(G,2)=2$.

\medskip

\begin{lemma}\label{lem:estimate}
Let $X\sim \Bin(n,p)$ with $np\le 1$. Then $Pr[X>0]>\frac{np}{2}$.
\end{lemma}

\begin{proof}
By Bonferroni's Inequality,
$$
Pr[X>0]\ge np -\binom{n}{2}p^2=np\left(1-\frac{(n-1)p}{2}\right)>\frac{np}{2}\,.
$$
\end{proof}

\medskip

We expose $G(n,p)$ in two stages: $G=G_1\cup G_2$, where $G_i\sim G(n,p_i)$, $p_1=\frac{1}{\sqrt{n\log n}}$, $p_2=\frac{C_1}{\sqrt{n\log n}}$ with $C_1$ being a large enough constant, to be set later. We will argue that w.h.p. $G_1$ contains two vertices $u_1,u_2$ infecting a set $U$ of size $k=\Theta(\log n)$. Then we will use $G_2$ and Theorem~\ref{thm:Janson} to argue that if $C_1$ is large enough then with high probability the set $U$ infects all of $V$ in $G_2$ and thus in $G$.

\begin{lemma}\label{thm:intermdiate}
Let $k=c_1\log n$,
where $0<c_1<1$ is a small enough constant.  Let $G_1$ be distributed as $G\left(n,\frac{1}{\sqrt{n\log n}}\right)$. Then with high probability there are two vertices in $G_1$ that infect a set of size $k$.
\end{lemma}

\begin{proof}
Initialize $V_0=V=[n]$. We describe an algorithm that has at most $\frac{n}{2k}$ iterations, indexed by $i=1, \ldots, \frac{n}{2k}$. Every iteration has at most $k-2$ steps, indexed by $j=3,\ldots, k$. We now describe iteration $i$.

Let $u_1,u_2$ be arbitrary vertices of $V_0$. For simplicity of the proof (and at the expense of requiring a smaller constant $c_1$ in the statement of the lemma), partition $V_0-\{u_1,u_2\}$ into $k-2$ sets $U_{i,1},\ldots,U_{i,k-2}$, each of size at least $\lfloor\frac{|V_0|-2}{k-2}\rfloor$. For $j=3,\ldots, k$, if there is a vertex $v_j\in U_{i,j-2}$ with at least two neighbors in $\{u_1,u_2,\ldots,u_{j-1}\}$, then set $u_j:=v_j$. Otherwise, abort iteration $i$, dump $\{u_1,\ldots,u_{j-1}\}$, update $V_0:=V_0-\{u_1,\ldots,u_{j-1}\}$, and move to iteration $i+1$.

Observe crucially that during iteration $i$ we have only exposed edges of $G_1$ touching $\{u_1,\ldots,u_{j-1}\}$, so the rest (i.e. edges whose both endpoints belong to $V_0$) are not exposed and fully retain their randomness. Also, at each iteration we dump less than $k$ vertices; since we perform at most $\frac{n}{2k}$ iterations, the size of $V_0$ is always at least $n/2$.

Let us now estimate the probability that the $i$-th iteration succeeds. When looking for $u_j$ inside this iteration,
the probability that such a vertex is found is at least the probability that there is a vertex in a set of size $n/2k$ having at least 2 neighbors in a set of size $j-1$, where the edge probability is $p_1$. The probability for a given vertex of $U_{i,j}$ to have at least two neighbors in the set of size $j-1$ can be estimated from below by $\binom{j-1}{2}\frac{p_1^2}{2}>\frac{(j-2)^2p_1^2}{4}$. Thus the probability that there is a required vertex in $U_{i,j}$ is at least $Pr[\Bin\left(\frac{n}{2(k-2)},\frac{(j-2)^2p_1^2}{4}\right)>0]$, and the latter is at least $\frac{n(j-2)^2p_1^2}{16k}$ by Lemma~\ref{lem:estimate}, as $\frac{n}{2(k-2)}\cdot \frac{(j-2)^2p_1^2}{4}<1$ (recall we assume $k=c_1\log n$ where $c_1$ is sufficiently small); this estimate is valid independently of what happened in the current iteration. Thus the probability that the $i$-th iteration succeeds is at least (using Stirling's approximation)
$$
\prod_{j=3}^k \frac{n(j-2)^2p_1^2}{16k}=\left(\frac{np_1^2}{16k}\right)^{k-2}\cdot ((k-2)!)^2
\ge \left(\frac{np_1^2(k-2)^2}{16e^2k}\right)^{k-2}\ge \left(\frac{np_1^2k}{200}\right)^k\,.
$$
Substituting the expressions for $k$ and $p_1$, we obtain that the above expression is at least $(c_1/200)^{c_1\log n}$, and this is more than $\frac{2k\log n}{n}$ for $c_1>0$ small enough.

Since we are ready to perform $\frac{n}{2k}$ iterations, with each being successful independently with probability at least $\frac{2k\log n}{n}$, w.h.p. one of them will indeed succeed -- resulting in a set of size $k$ which can be infected by two vertices in it.
\end{proof}

\medskip

The equality $m(G,2)=2$ now follows from Theorem~\ref{thm:Janson}. Namely, in $G_2$ there is an active set $S$ of cardinality $c_1 \log n$ (generated by choosing two ``correct" vertices to start the process in $G_1$). Hence when $p_2=\frac{C_1}{\sqrt{n\log n}}$ for $C_1 > \frac{1}{\sqrt{c_1}}$, we get that with high probability $S$ is a contagious set in $G_2$. 

We now deal with the number of generations. We first consider the upper bound. To this end, we upper bound the number of generations until activation of the contagious set constructed in Lemma~\ref{thm:intermdiate}. We analyze first the number of generations it takes to infect $k$ vertices in the infection process occurring in $G_1$. For this, consider the following random directed graph which we denote by $H_{2,k}$. There are $k$ vertices numbered from~1 to $k$. Each vertex $i \ge 3$ has two outgoing arcs to two random vertices of index less than $i$. It is implicit in the proof of Lemma~\ref{thm:intermdiate}, that the length of the longest directed path of $H_{2,t}$ is an upper bound on the number of generations, which we denote by $l(H_{2,k})$. The parameter $l(H_{2,k})$ was studied in several previous works (e.g.,~\cite{Arya,Depth}) and shown to be of order $\Theta(\log k)$. Here we present a simple self contained proof that $l(H_{2,k})=O(\log k)$ (the leading constant in the $O$-term in our proof is not optimal).

\begin{lemma}\label{lem:diameter}
With high probability, $l(H_{2,k})$ is at most $40 \log k$.
\end{lemma}
\begin{proof}
Let $0 < \rho < 1$ be a constant to be optimized later. Call an arc $(i,j)$ in $H_{2,k}$ {\em good} if $j \le \rho i$ and {\em bad} if $j > \rho i$ (note that necessarily $j < i$). A path can have at most $g$ good arcs, where $g$ is largest number satisfying $k\rho^g \ge 1$. Given a vertex $i$, the probability that a random outgoing arc is bad is at most $(1 - \rho)$. For arbitrary $t \ge 2g$, let us upper bound the probability that there is a path of length $t$. There are (less than) $k$ possible starting points. From each vertex there are two outgoing arcs to choose. So the number of candidate paths is at most $k2^t$. For each candidate path, there are $\sum_{i=1}^g {t \choose i} \le 2^t$ possible locations for the good arcs. For the rest of the arcs to be bad, the probability is at most $(1 - \rho)^{t-g} \le (1 - \rho)^{t/2}$. Hence the probability that some candidate path actually reaches length $t$ is at most $k2^{2t}(1 - \rho)^{t/2}$.

Choose $\rho = \frac{19}{20}$. Then $g \simeq 20\ln k$, and we can choose $t = 40\ln k$. For these parameters,  $k2^{2t}(1 - \rho)^{t/2} = k \frac{2^{40 \ln k}}{20^{10\ln k}} = k (\frac{4}{5})^{10\ln k} = o(1)$. Hence w.h.p. $l(H_{2,k})$ does not exceed $40\ln k$.
\end{proof}

Lemma~\ref{lem:diameter} implies that with high probability the number of generations until $B_0$ infects a set of size $c_1 \log n$ in $G_1$ is at most $O(\log k)=O(\log \log n)$.
Thereafter, Theorem~\ref{thm:Janson} implies that with high probability all the vertices in $G_2$ are infected within $O(\log \log n)$ generations.

Now we establish a lower bound on the number of generations.
We first claim that with high probability no set of size $k \geq \log n$ can infect too many vertices in a single round.

\begin{lemma}\label{lemma:growth}
Suppose that $p \leq \frac{1}{\sqrt{2 e n}}$. Then, with high probability every set of size $k \geq \log n$ in $G(n,p)$ infects (in one round) a set of size smaller than $k^2$.
\end{lemma}
\begin{proof}
Given that a set $S$ of size $k$ is active, the probability that a vertex outside $S$ is infected by $S$ in one round is $Pr(\Bin(k,p)\geq 2) \le (pk)^2$.
Therefore, the probability for a fixed $k$ there is a set of size $k$ infecting $k^2$ additional vertices in one round is at most
$${n \choose k} \cdot {n \choose k^2}(p^2 k^2)^{k^2},$$
Which can be upper bounded (when $k \ge \log n$) by
$$\left(\frac{en}{k}\right)^k\left(\frac{1}{2}\right)^{k^2}=o(1/n).$$
Taking a union bound over all $k \geq \log n$ concludes the proof.
\end{proof}

Lemma~\ref{lemma:growth} implies that with high probability, for every contagious set of size $\log n$, the number of generations required to infect $G(n,p)$ for $p \leq \frac{1}{\sqrt{2 e n}}$ is at least $\log \log n -\log \log \log n$. The same must hold for contagious sets of size~2 (because every contagious set of size~2 is contained in a contagious set of size $\log n$). This concludes the analysis of the number of generations and the proof of Theorem~\ref{thm:threshold} (for the case $r = 2$).

\section{Generalizing the results for $r>2$}\label{sec:generalization}

In this section we study the case where the threshold of every vertex is $r$, where $r>2$ is a fixed constant. As the proofs are similar to the $r=2$ case, we sketch the main ideas without going into every detail.

\subsection{The asymptotic value of $m(G,r)$ in $G(n,p)$}

Here we explain that a similar reasoning to the case $r=2$ implies that in $G(n,p)$, w.h.p. $m(G,r)\sim \frac{n}{d^{\frac{r}{r-1}}\log d}$. We begin by discussing the upper bound.
First we have the following lemma.

\begin{lemma}\label{lemma:partial_infecr}
Let $r>2$ be a fixed integer (independent of $n$).
For $d\ge 100 r$, a random graph $G\sim G(n,p)$ is w.h.p. such that activating any set of size at least $n/2$ infects all but at most $n/d^3$ vertices.
\end{lemma}
\begin{proof}
This follows from the fact that if a set $U$ is disjoint from a set $\langle A_0 \rangle$, then $U$ can have at most $r|U|$ neighbors in $\langle A_0 \rangle$. Using the assumption that $\frac{n}{d^3} \leq |U| \leq n/2$ and choosing $d_0 > 100r$, implies the lemma along similar lines to Lemma~\ref{lemma:partial_infec} -- details omitted.
\end{proof}

\medskip

Our goal is to infect a set of size $\frac{C_1n}{d^{\frac{r}{r-1}}}$ where $C_1$ is large enough.  Then by applying Theorem~\ref{thm:Janson}, we conclude that with high probability at least $n/2$ vertices are activated. Finally, using Lemma~\ref{lemma:partial_infecr} we will be able to deduce that with high probability all of $G$ infected. To achieve this goal, it suffices to make some modest changes to the algorithm presented in Theorem~\ref{thm:main}. The main difference is that now we look in the $i$th iteration for large connected components in the set of all vertices having $r-1$ neighbors in $C_{i-1}$. As in the proof of Theorem~\ref{thm:main} we initially set $\ell=\log \log d$ and $B_0=C_0=D_0$ to be a fixed subset of $[n]$ of size $\frac{C_1n}{d^{\frac{r}{r-1}}\log d}$ and run for $\ell$ iterations, where in every iteration $\ell$ drops down by $1$ -- terminating once $\ell=0$. Specifically, \newline
{\bf Iterating:} We aim to get a set $C_i$ of size
$$
|C_i|=\frac{C_1n}{d^{\frac{r}{r-1}}2^{\ell-i}}\,.
$$
Given a set $C_{i-1}$, we find a subset $B_i$ of vertices in $V\setminus\bigcup_{j=0}^{i-1}B_j$, all having at least $r-1$ neighbors in $C_{i-1}$, and
$$
b_i:=|B_i|=\frac{C_1^{r-1}}{2^{(\ell-i+1)(r-1)}\cdot d\cdot(r-1)^{r-1}}\, \frac{n}{2} =: n_0\,.
$$
Since the probability a vertex (disjoint from $C_{i-1}$) has at least $r-1$ neighbors in $C_{i-1}$ is asymptotically equal to  $\binom{|C_{i-1}|}{r-1}p^{r-1}\ge\left(\frac{|C_{i-1}|}{r-1}p\right)^{r-1}$, we get using the Chernoff bound that the size of $B_i$ is indeed lower bounded by $\left(\frac{|C_{i-1}|}{r-1}p\right)^{r-1}2(n-o(1))/3 > b_i$ with probability at least $1-\exp\left(-O(\frac{n}{d (\log d)^{r-1}})\right)$, where terms depending only on $r$ are treated as constants (as we assume $r$ is a constant not depending on $n$). It is therefore straightforward to verify that with high probability for all $i \in [\ell]$ it holds that $|B_i| \geq b_i$.

Analogously to the $r=2$ case, if a set $S$ is connected, and every vertex in $S$ has at least $r-1$ active neighbors, it suffices to activate a single vertex in $S$ in order
to infect the whole of $S$. Hence we estimate the number of ``large" connected components in $B_i$.

Set
$$
s_i =\frac{\log d}{(\ell-i+1)r^2}\,
$$

Apply Lemma~\ref{lemma:subcritical} to $G[B_i]$ with parameters
$$
n_0,\quad q=\frac{d}{n}=\frac{C_1^{r-1}}{2^{(\ell-i+1)(r-1)+1}(r-1)^{r-1}n_0},\quad k=s_i\,.
$$
We need to verify:
$$
\left(\frac{C_1^{r-1}}{3\cdot 2^{(\ell-i+1)(r-1)+1}(r-1)^{r-1}}\right)^{s_i}\frac{n_0}{4}\ge \frac{C_1n}{d^{\frac{r}{r-1}}2^{\ell-i}}\,,
$$
which amounts to
$$
\left(\frac{3\cdot 2^{(\ell-i+1)(r-1)+1}(r-1)^{r-1}}{C_1^{r-1}}\right)^{s_i}\le \frac{n_0d^{\frac{r}{r-1}}2^{\ell-i}}{4C_1n}
= \frac{C_1^{r-2}d^{\frac{1}{r-1}}2^{\ell-i-2}}{2^{(\ell-i+1)(r-1)+1}(r-1)^{r-1}}\ .
$$
This follows from $\left(\frac{6}{C_1^{r-1}}\right)^{s_i}\le \frac{C_1^{r-2}}{16(r-1)^{r-1}}$ and $2^{(\ell-i+1)(r-1)s_i}\le
\frac{d^{\frac{1}{r-1}}}{2^{(\ell-i+1)(r-2)+4}}$. The former inequality is valid when $C_1>6 \cdot 17(r-1)^{r-1}$ and large enough $d$ (e.g., $d$ such that $s_i>1$). For the latter inequality, we need to satisfy
$$
s_i\le \frac{\frac{1}{r-1}\log d -(\ell-i+1)(r-2)-4}{(\ell-i+1)(r-1)}$$

$$= \frac{\log d}{(\ell-i+1)(r-1)^2}-\frac{r-2}{r-1}-\frac{4}{(\ell-i+1)(r-1)}
$$
-- which is indeed valid for our choice of $s_i$.

From the calculations outlined in the paragraph above and Lemma~\ref{lemma:subcritical}, the probability that $|C_i|< \frac{C_1n}{d^{\frac{r}{r-1}}2^{\ell-i}}$ is upper bounded by
$$e_i:=\exp \left(- \frac{C_1n}{s_id^{\frac{r}{r-1}}2^{\ell-i}}\right)=\exp \left(- \frac{C_1n(\ell-i+1)r^2\log d}{d^{\frac{r}{r-1}}2^{\ell-i}}\right).$$ Hence the probability that a failure will occur in one of the rounds is upper bounded by $\sum_{i=1}^{\ell} e_i$. Analogous reasoning to the $r=2$ case implies that $$\sum_{i=1}^{\ell} e_i=\exp\left(-\Theta\left(\frac{n \log \log d}{d^{\frac{r-1}{r}}\log^2 d}\right)\right).$$  Substituting $d = o\left(\left(\frac{n \log \log n}{\log^2 n}\right)^{\frac{r-1}{r}}\right),$ we have that the probability there exists $i \in [\ell]$ such that $|C_i|< \frac{C_1n}{d^{\frac{r}{r-1}}2^{\ell-i}}$ is $o(1)$.

The total size of the seed of Stage I is then
\begin{eqnarray*}
&&\frac{C_1n}{d^{\frac{r}{r-1}}\log d}+\sum_{i=1}^{\ell}\frac{C_1n}{d^{\frac{r}{r-1}}2^{\ell-i}s_i}
=\frac{C_1n}{d^{\frac{r}{r-1}}\log d}+\frac{C_1n}{d^{\frac{r}{r-1}}\log d}\sum_{i=1}^{\ell}\frac{\ell-i+1}{2^{\ell-i}}\\
&=& O\left(\frac{n}{d^{\frac{r}{r-1}}\log d}\right)\,.
\end{eqnarray*}

This concludes the proof that for $d$ satisfying the condition in Theorem~\ref{thm:random}, with high probability $m(G,r) \leq O\left(\frac{n}{d^{\frac{r-1}{r}}\log d}\right)$.

Now we turn to the lower bound. We shall use the following auxiliary Lemma:
\begin{lemma}
Let $\alpha=3$ and $\beta = r - \frac{r-1}{\log d}$. Set $t = \frac{n}{\alpha d^{\frac{r}{r-1}}}$ and assume $d=np=o(n)$ is larger than an appropriate constant $d_0$ (that may depend on $r$). Then with high probability no set of vertices of size $t$ spans $\beta t$ edges.
\end{lemma}
\begin{proof}
Using the equality $p=\frac{1}{\alpha d^{\frac{1}{r-1}}t}$ we conclude that the probability that $G(n,p)$ contains a set of size $t$ that spans at least $\beta t$ edges is upper bounded by
$$
\binom{n}{t}\binom{\binom{t}{2}}{\beta t}p^{\beta t}\le
(e\alpha d^{\frac{r}{r-1}})^t \left(\frac{et}{2\beta}\right)^{\beta t} \left(\frac{1}{\alpha d^{\frac{1}{r-1}}t}\right)^{\beta t} = \left(\frac{e^{\beta + 1}d^{1/\log d}}{\alpha^{\beta - 1}2^{\beta}\beta^{\beta}}\right)^t.$$ It can be verified that the latter expression is $o(1)$.
\end{proof}
\begin{corollary}
Let $G$ be distributed as $G(n,p)$, where $d_0<d=np \ll n^{-1/r}$ and $d_0$ is a large enough constant that may depend on $r$ but not on $n$. Then with high probability
$m(G,r) > \frac{(r-1)n}{3r d^{\frac{r}{r-1}} \log d}$.
\end{corollary}
\begin{proof}
Suppose there exists a contagious set $A_0$ of size $t_0=\frac{(r-1)n}{3r d^{\frac{r}{r-1}} \log d}$. Setting $t = \frac{n}{3 d^{\frac{r}{r-1}}}$, we get that $A_0$ together with the first $t-t_0$ infected vertices, would produce a set $A$ of size $\frac{n}{3d^{\frac{r}{r-1}}}$ spanning at least $r(t-t_0)=(r-\frac{r-1}{\log d})t$ edges. As we have just shown, w.h.p. such a set $A$ does not exist. This concludes the proof.
\end{proof}

\subsection{The threshold for $m(G,r)=r$}

It turns out that the threshold for the emergence of a contagious set of size $r$ in $G(n,p)$ is $p\sim (n\log^{r-1} n)^{-1/r}$. We begin by proving that contagious sets of size $r$ are unlikely to exist when $p=c (n\log^{r-1} n)^{-1/r}$, for some appropriate constant $c>0$.

\begin{lemma}
Suppose that $p\leq c (n\log^{r-1} n)^{-1/r}$ for some $c>0$ that is sufficiently small. Then with high probability $m(G,r)>r$, when $G$ is distributed as $G(n,p)$.
\end{lemma}
\begin{proof}
By a similar reasoning to the case  $r=2$, if $G(n,p)$ has a contagious set of size $r$ then for every $1<t\leq n-r$ it has a set of $t+r$ vertices spanning at least $rt$ edges. The probability that such a set exists is upper bounded by
$${n \choose r}{n \choose t}{(t+r)^2/2 \choose rt}p^{rt}.$$
For large enough $t$, the expression above can be upper bounded by
$$n^r\left(\frac{e n}{t}\right)^t\left(\frac{e(t+r)^2p}{2rt}\right)^{rt}<n^r\left(n\frac{ (etp/r)^r}{t}\right)^{t},$$
where we used the fact that $(t+r)^2<2t^2$ for large enough $t$.
Setting $t=\log n$ and substituting the value of $p$, we can upper bound the expression above by
$n^r (c')^t$ for some $c'>0$ that tends to $0$ as $c\rightarrow 0$. Hence, taking $c$ to be sufficiently small we can ensure that $(c')^{\log n}<n^{-{r+1}}$, implying that the probability there exists a contagious set of size $r$ is at most $1/n$. This concludes the proof of the Lemma.
\end{proof}

\medskip

We now proceed and prove that if $p>C (n\log^{r-1} n)^{-1/r}$ for a suitable constant $C$, then with high probability there is a contagious set of size $r$ in $G(n,p)$.

\begin{theorem}
Suppose that $p>C (n\log^{r-1} n)^{-1/r}$, where $C$ is a sufficiently large constant that may depend on $r$. Let $G$ be distributed as $G(n,p)$. Then with high probability $m(G,r)=r$.
\end{theorem}
\begin{proof}
It suffices to prove that for $p_1=(n\log^{r-1} n)^{-1/r}$, a random graph $G_1\sim G(n,p_1)$ is typically such that activating appropriately chosen $r$ vertices will infect $c_1 \log n$ vertices. Thereafter, exposing the remaining edges of $G=G(n,p)$ with probability $p_2=C_2(n\log^{r-1} n)^{-1/r}$, where $C_2$ is a large enough constant, and using Theorem~\ref{thm:Janson} implies that that the whole of $G$ gets infected with high probability.

We use ideas similar to those appearing in Lemma~\ref{thm:intermdiate}. Let $k=c_1\log n$
where $0<c_1<1$ is constant that will be determined later. Run the same iterative procedure as in Lemma~\ref{thm:intermdiate}, but choose initially a set $I$ of $r$ vertices from $V_0$. Now partition $V_0\setminus I$ to $k-r$ sets each of size at least $\lfloor\frac{|V_0|-r}{k-r}\rfloor$. We now run an iterative procedure identical to the one in Lemma~\ref{thm:intermdiate}, but search for a vertex in $v_j\in U_{i,j-r}$ having at least $r$ neighbors in $\{u_1,u_2,\ldots,u_{j-1}\}$ (in the procedure in Lemma~\ref{thm:intermdiate}, $r=2$). If found, set $u_j:=v_j$.
If such a $u_j$ is not found, we stop iteration $j$, delete $\{u_1,\ldots,u_{j-1}\}$ and update $V_0:=V_0-\{u_1,\ldots,u_{j-1}\}$.
The probability that the $j$th iteration succeeds can be lower bounded by $\binom{j-1}{r}\frac{p_1^r}{2}\geq \left(\frac{(j-r)p_1}{r}\right)^r/2$ (recall that $r$ is a fixed constant and $p=o(1)$). Assuming $c$ to be a sufficiently small constant that may depend on $r$, we reason that $\Pr[\Bin(\frac{n}{2(k-r)},\left(\frac{(j-r)p_1}{r}\right)^r/2)>0]\geq \frac{n}{2(k-r)} \cdot \left(\frac{(j-r)p_1}{r}\right)^r/2)$.
Therefore, the probability the $i$th iteration succeeds is at least
$$
\prod_{j=r+1}^k \frac{np_1^r}{2r^r k}\cdot (j-r)^r=\left(\frac{np_1^r}{2r^rk}\right)^{k-r}\cdot \left((k-r)!\right)^r
\ge \left(\frac{np_1^r(k-r)^r}{2er^rk}\right)^{k-r}\,.
$$
Choosing $c_1$ to be small enough and plugging in $k$ and $p$, the aforementioned probability can be lower bounded by
$(c_1/10)^{c_1\log n}>n^{-1/3}$. The rest of the argument is essentially identical to that of Lemma~\ref{thm:intermdiate}.
\end{proof}

Similar arguments to those presented in Section \ref{sec:thresholdtwo} imply that the number of generations for a contagious set as above to activate $G$ is $\Theta(r \log \log n)$. We omit the details.

\section{Conclusions}\label{sec:concl}

The discussion below concerns the case $r=2$.

Theorems~\ref{thm:random} and~\ref{thm:threshold} both show that the smallest contagious set has size $\Theta\left(\frac{n}{d^2\log d}\right)$ w.h.p., but address two different ranges of degrees. The negative results (nonexistence of small contagious sets) in both theorems are based essentially on the same argument (lower bounds on the size of the smallest subgraph of average degree $4 - O(\frac{1}{\log d})$). However, our proofs of the upper bounds in the two theorems are based on different principles. Our proof for Theorem~\ref{thm:random} is based on an algorithm that performs $\log\log d$ iterations, where in every iteration additional vertices are designated as seeds. Such a proof produces a contagious set of size at least $\Omega(\log\log d)$ (in fact, our proof of Proposition~\ref{prop:second} requires values of $d$ for which the contagious set is even larger), and hence is inappropriate for Theorem~\ref{thm:threshold}, in which the total number of seeds allowed is only~2. In contrast, the proof of Theorem~\ref{thm:threshold} is based on examining disjoint pairs of vertices until some pair is found to be contagious. Disjointness implies that the proof examines much fewer than $n$ candidate sets for being contagious. Such a proof is inappropriate for Theorem~\ref{thm:random}, because for the range of degrees considered in Theorem~\ref{thm:random} a random set of size $\Theta\left(\frac{n}{d^2\log d}\right)$ has probability much less than $1/n$ of being contagious.

In this work we did not handle two ranges of degrees. One is when $d$ is a large constant. It is not difficult to extend Theorem~\ref{thm:random} also to the case of large constant degrees. This range of degrees is omitted from the current work mainly for the reasons of simplicity, as we are using Theorem~\ref{thm:Janson} as a blackbox, and that Theorem requires $d$ to be super-constant.  The more challenging range of parameters omitted from our paper is when $d = o\left(\sqrt{\frac{n}{\log n}}\right)$ but still too large for Theorem~\ref{thm:random} to apply. It would be interesting to prove that the smallest contagious set has typically size $\Theta\left(\frac{n}{d^2\log d}\right)$ also in this regime.


The positive results in Theorem~\ref{thm:threshold} implicitly establish one specific average degree $d = \Theta\left(\sqrt{\frac{n}{\log n}}\right)$ that suffices with high probability for two related problems: one is the existence of a contagious set of size~2, and the other is the existence of such a set for which the number of generations is $O(\log\log n)$ (which is best possible up to constant multiplicative factors). For both problems, this value of $d$ is best possible up to constant factors. Nevertheless, it would be interesting to determine whether there is some $d' < d$ for which with high probability there is a contagious set of size $2$, but every contagious set of size $2$ requires more than $O(\log\log n)$ generations.

Theorem~\ref{thm:random} does not explicitly address the number of generations. An upper bound on the number of generations implicit in our proof of Theorem~\ref{thm:random} is $O(\log d \log\log d)$, but we doubt that it is tight.


\begin{thebibliography}{99}

\bibitem{Aizenman}
M. Aizenman and J. L. Lebowitz.
\newblock Metastability effects in bootstrap percolation.
\newblock {\em J. Phys. A.} 21: 3801--3813, 1988.

\bibitem{Amini}
H. Amini.
\newblock Bootstrap percolation and diffusion in random graphs with given
vertex degrees.
\newblock {\em Electronic Journal of Combinatorics}, 17: Research paper 25, 2010.

\bibitem{AF}
H. Amini and N. Fountoulakis.
\newblock What I tell you three times is true: bootstrap percolation in small worlds.
\newblock In {\em Proceedings of the 8th Workshop on Internet and Network Economics} (WINE '12) (P. Goldberg, Ed.), Lecture Notes in Computer Science, 7695: 462--474, 2012.

\bibitem{Arya}
S. Arya, M. Golin, and K. Mehlhorn.
\newblock  On the expected depth of random circuits.
\newblock {\em Combinatorics Probability and Computing} 8: 209–-228, 1999.

\bibitem{BB}
J. Balogh and B. Bollob\'as.
\newblock Bootstrap percolation on the hypercube,
\newblock {\em Probabilty Theory and Related Fields}, 134:  624--648, 2006.

\bibitem{Lattice}
J. Balogh, B. Bollob\'as, H. Duminil-Copin and R. Morris.
\newblock The sharp threshold for bootstrap percolation
in all dimensions.
\newblock {\em Trans. Amer. Math. Soc.}, 364: 2667--2701, 2012.

\bibitem{peres}
J. Balogh, Y. Peres and G. Pete.
\newblock Bootstrap percolation on infinite trees and non-amenable groups.
\newblock {\em Combinatorics, Probability and Computing}, 15: 715--730, 2006.

\bibitem{square}
J. Balogh and G. Pete.
\newblock Random disease on the square grid.
\newblock {\em Random Structures and Algorithms}, 13: 409--422, 1998.

\bibitem{BP}
J.~Balogh and B.~Pittel.
\newblock Bootstrap percolation on the random regular graph.
\newblock {\em Random Structures and Algorithms}, 30: 257--286, 2007.

\bibitem{Boll}
B. Bollob\'as.
\newblock \emph{Random graphs}.
\newblock Cambridge Stud. Adv. Math. 73, Cambridge University Press,
Cambridge, 2001.

\bibitem{Bolo1}
B. Bollob\'as, C. Holmgren, P.J. Smith, and A.J. Uzzell.
\newblock The time of bootstrap percolation for dense initial sets.
\newblock {\em Annals of Probability}, 42, 1337--1373, 2014.

\bibitem{Bolo2}
B. Bollob\'as, P.J. Smith, and A.J. Uzzell.
\newblock The time of bootstrap percolation with dense initial sets for all thresholds.
\newblock {\em Random Structures and Algorithms} 47, 1--29, 2015.

\bibitem{Bollo3}
B. Bollob\'as and A. Thomason.
\newblock Threshold functions.
\newblock {\em Combinatorica} 7, 35–-38, 1987.


\bibitem{Expanders}
A. Coja-Oghlan, U.~Feige, M. Krivelevich and D.~Reichman.
\newblock Contagious sets in expanders.
\newblock {\em SODA}, 1953--1987, 2015.

\bibitem{Chal}
J. Chalupa, P. L. Leath and G. R. Reich.
\newblock Bootstrap percolation on a Bethe lattice.
\newblock {\em J. Phys. C: Solid State Phys.}, 12, p L31, 1979.

\bibitem{Chen09}
N. Chen.
\newblock On the approximability of influence in social networks.
\newblock {\em SIAM Journal of Discrete Math}, 23: 1400--1415, 2009.


\bibitem{friedman03}
J. Friedman.
\newblock A proof of Alon's second eigenvalue conjecture and related problems.
\newblock {\em Memoirs of the AMS}, 910, 2008.

\bibitem{frieze}
A. Frieze and C. McDiarmid.
\newblock Algorithmic theory of random graphs.
\newblock {\em Random Structures and Algorithms}, 10: 5--42, 1997.


\bibitem{polo}
D. Freund, M. Poloczek and D. Reichman.
\newblock Contagious Sets in Dense Graphs.
\newblock arXiv preprint arxiv:1503.00158, 2015.

\bibitem{Janson}
S. Janson.
\newblock On percolation in random graphs with given degree sequence.
\newblock {\em Electronic Journal of Probability}, 14: 86--118, 2009.

\bibitem{JanLuc}
S. Janson, T. {\L}uczak, T. Turova and T. Vallier.
\newblock Bootstrap percolation on the random graph {$G_{n,p}$}.
\newblock {\em Annals of Appied Probability}, 22: 1989--2047, 2012.

\bibitem{MU}
M. Mitzenmacher and E. Upfal.
\newblock {\em Probability and Computing, randomized algorithms and probabilistic analysis}.
\newblock Cambridge University Press, Cambridge, 2005.

\bibitem{Re12} D. Reichman.
\newblock New bounds for contagious sets.
\newblock {\em Discrete Math.}, 312: 1812--1814, 2012.

\bibitem{Tomba}
G.-P. Scalia-Tomba.
\newblock Asymptotic final-size distribution for some chain-binomial processes.
\newblock {\em Adv. in Appl. Probab.}, 17: 477--495, 1985.


\bibitem{Depth}
T. Tsukiji and F. Xhafa.
\newblock On the depth of randomly generated circuits.
\newblock {\em ESA}, 208-–220, 1996.

\bibitem{Enter}
A. C. D. van Enter.
\newblock Proof of Straley's argument for bootstrap percolation.
\newblock {\em J. Stat. Phys.}, 48: 943--945, 1987.

\bibitem{Vallier}
T. Vallier.
\newblock Random graph models and their applications.
\newblock Ph.D. thesis, Lund University, 2007.

\end{thebibliography}
\end{document}